\documentclass[12pt]{article}
\usepackage{amssymb, latexsym, amsthm, amsmath, amsfonts}
\usepackage{geometry}

\newtheorem{theorem}{Theorem}
\newtheorem*{theorem*}{Theorem}
\newtheorem{lemma}{Lemma}[section]
\newtheorem*{lemma*}{Lemma}
\newtheorem{corollary}[lemma]{Corollary}
\newtheorem*{corollary*}{Corollary}
\newtheorem{proposition}[lemma]{Proposition}
\newtheorem*{proposition*}{Proposition}
\newtheorem*{remark*}{Remark}

\newcommand{\PA}{\mathsf{PA}}
\newcommand{\EA}{\mathsf{EA}}
\newcommand{\ISi}{\mathsf{I}\Sigma}
\newcommand{\IPi}{\mathsf{I}\Pi}
\newcommand{\BSi}{\mathsf{B}\Sigma}

\newcommand{\Rfn}{\mathsf{Rfn}}
\newcommand{\RFN}{\mathsf{RFN}}
\newcommand{\True}{\mathsf{True}}

\newcommand{\Prf}{\mathsf{Prf}}
\newcommand{\Con}{\mathsf{Con}}

\newcommand{\ul}{\ulcorner}
\newcommand{\ur}{\urcorner}

\newcommand{\imp}{\rightarrow}
\newcommand{\eqv}{\leftrightarrow}

\newcommand{\gn}[1]{\ulcorner #1 \urcorner}
\renewcommand{\phi}{\varphi}
\newcommand{\num}{\underline}
\newcommand{\nat}{\mathbb{N}}
\newcommand{\omegacon}{\omega\textsf{-}\Con}
\newcommand{\Def}{\mathsf{Def}}
\newcommand{\DRfn}{\mathsf{DRfn}}

\title{Local reflection, definable elements and 1-provability}
\author{Evgeny Kolmakov\\
        Lomonosov Moscow State University\\
        Moscow, Russia\\
        \texttt{kolmakov.evgn@gmail.com}
}

\begin{document}

\maketitle

\begin{abstract}

In this note we study several topics related to the schema of local reflection $\Rfn(T)$
and its partial and relativized variants. Firstly, we introduce the principle of uniform reflection with $\Sigma_n$-definable parameters,
establish its relationship with the relativized local reflection principles and corresponding versions of induction with definable parameters. 
Using this schema we give a new model-theoretic proof of the $\Sigma_{n+2}$-conservativity of uniform $\Sigma_{n+1}$-reflection over 
relativized local $\Sigma_{n+1}$-reflection.
We also study the proof-theoretic strength of Feferman's theorem, i.e., the assertion
of $1$-provability in $S$ of the local reflection schema $\Rfn(S)$, and its generalized versions. 
We relate this assertion to the uniform $\Sigma_2$-reflection schema and, in particular, 
obtain an alternative axiomatization of $\ISi_1$. 

{\vskip 1em}

\end{abstract}

\section{Introduction}

In the original formulation of G{\"o}del's first incompleteness theorem
no semantic concepts, such as \emph{soundness}, were used. Instead, G{\"o}del used a rather strong syntactic notion of $\omega$-consistency. 
However, in the contemporary formulations of G{\"o}del's incompleteness theorems a weaker condition of \emph{$\Sigma_1$-soundness}
is often used instead of $\omega$-consistency.

On the syntactic level, various semantic notions of soundness are usually translated into the so-called \emph{reflection principles}.
These principles typically express the following form of soundness
$$
T \vdash \phi \Longrightarrow \nat \models \phi,
$$
with different conditions being imposed on the formula $\phi$ and on the notion of provability in $T$.
Generalizing G{\"o}del's consistency assertion $\Con(T)$, 
principles of this kind provide plenty of true sentences which are nevertheless unprovable in $T$.

Reflection principles were shown to be a convenient tool for the analysis of formal theories
by demonstrating that many other principles, e.g., induction or different forms of consistency, can be expressed as some form of reflection
and due to the \emph{unboundedness theorems} of G.~Kreisel and A.~L{\'e}vy \cite{KrLe}, which allow to conclude that 
one formal theory $T$ cannot be axiomatized over another formal theory $U$ by the arithmetical sentences of a certain logical complexity,
whenever $T$ proves the corresponding form of reflection for $U$. 
For a survey of the results on reflection principles, see C.~Smory{\'n}ski \cite{Smo77b} and L.~Beklemishev \cite{Bek05en}.

Various kinds of reflection principles and consistency assertions for $\PA$
were studied by R.~Kaye and H.~Kotlarski in \cite{KayKot} from the point of view of ACT-extensions of models 
(extensions constructed by the means of the arithmetized completeness theorem).
More specifically, the authors characterize these principles as the first-order theories of the class of models of $\PA$ 
having ACT-extensions with certain model-theoretic properties.

One of the principles considered in \cite{KayKot} was the principle of definable reflection $\DRfn(\PA)$,
which is the restriction of $\RFN(\PA)$ to formulas asserting definability of an element, i.e., formulas of the form 
$\phi(x) \land \forall y\left(\phi(y) \imp y = x\right)$.
This principle, however, turns out to be equivalent to $\RFN(\PA)$.

The idea of considering intermediate principles between local (only standard parameters) 
and uniform (no restriction on parameters) schemata by allowing restricted use of parameters provides 
a way to devise new schemata or at least give alternative, sometimes more useful, characterization of the existing principles.
In this respect, let us mention the work of A.~Cord{\'o}n-Franco, A.~Fern{\'a}ndez-Margarit and F.~F.~Lara-Mart{\'i}n \cite{CordFr}, 
where this idea is applied to the induction schemata. 
It is shown that parameter-free induction schema is equivalent to its variants with parameters restricted 
to the certain types of definable elements (cf. Proposition \ref{prop:cordfr}).

Continuing this line of research, we introduce a refined version of the definable reflection principle, 
namely, the schema of uniform reflection with $\Sigma_n$-definable parameters and 
show it to be equivalent to the corresponding form of relativized local reflection. 
This fact is then used to give a rather short model-theoretic proof of the 
$\Sigma_{n+2}$-conservativity of uniform $\Sigma_{n+1}$-reflection over 
relativized local $\Sigma_{n+1}$-reflection (Theorem \ref{th:s2conserv}).
We also establish the connection between this version of reflection principles and 
the induction schemata with definable parameters introduced in \cite{CordFr} (Corollary \ref{col:cordfr}).

Let us also mention the following questions regarding the relationship between local $\omega$-consistency $\omegacon(\PA)$
and local reflection $\Rfn(\PA)$, which were listed as open in \cite{KayKot}.

\medskip
\noindent {\bf Problem 8.1.} Do there exists models $M \models \PA + \Rfn(\PA) + \neg \omegacon(\PA)$? 

\medskip
\noindent {\bf Problem 8.3.} Is it true that $\PA + \Rfn(\PA) \not \vdash \omegacon^{\mathsf{Th}}(\PA)$?

\medskip
\noindent {\bf Problem 8.4.} Over $\PA$, does $\omegacon(\PA)$ imply $\Con(\PA + \Rfn(\PA))$?

\medskip \noindent For the formal definitions of the $\omega$-consistency and reflection principles mentioned in these questions see Section 2.

We show that all of these three questions have positive answer for a wide class of theories $T$ and give a quick solution based on Smory{\'n}ski's characterization 
of the schema $\omegacon(T)$ and its restricted variants in terms of partial reflection principles for $T$ (see Theorem \ref{th:ocon-s2rfn}). 
Let us note that for the case of $\PA$ these problems have already been solved by V.~Yu.~Shavrukov (unpublished), see Remark in Section 4 for a further discussion. 
We also prove the schema $\omegacon^{\mathsf{Th}}(T)$ to be equivalent to $\RFN(T)$ (Theorem \ref{th:oconth}).

Furthermore, we study the assertion (so-called Feferman’s theorem \cite[Theorem~4.5]{Fef62}) 
of $1$-provability in $S$ of the local reflection schema $\Rfn(S)$ and its generalized versions.
We relate this assertion to the usual forms of reflection (Theorem~\ref{th:rfnsi2}), 
which, in particular, allows to give an alternative axiomatization of the theory $\ISi_1$ (Corollary~\ref{cor:isi1rfn}).

The paper is organized as follows. Section 2 introduces the basic notions and notation used throughout this note.
In Section 3 we introduce and study the schema of uniform reflection with $\Sigma_n$-definable parameters.
In Section 4 we give a quick solution to the problems listed above and discuss the proof-theoretic strength of Feferman's theorem.

\section{Preliminaries}
In this note we consider first-order theories in the language of arithmetic. As our basic theory we take \emph{Elementary arithmetic} $\EA$
(sometimes denoted as $\mathsf{I}\Delta_0(\exp)$), that is, the first-order theory formulated in the language $0, (\cdot)', +, \times$
extended by the unary function symbol $\exp$ for the exponentiation function $2^x$.
It has the standard defining axioms for these symbols and the induction schema for all \emph{elementary formulas} 
(we also call such formulas \emph{bounded}), 
i.e., formulas in the language with exponent containing only bounded (by the terms in the language with exponent) quantifiers.
The theory $\EA^+$ is defined to be an extension of $\EA$ with the axiom asserting the totality of the superexponential function $2^x_y$.

We define classes $\Sigma_0$ and $\Pi_0$ to be the classes of all elementary (bounded) formulas. 
After that the classes $\Sigma_n$ and $\Pi_n$ of arithmetical hierarchy are defined in a standard way for all $n \geqslant 0$.

If we allow induction for all arithmetical formulas, the resulting theory is \emph{Peano arithmetic} denoted by $\PA$.
For a fixed class of arithmetical formulas $\Gamma$ the fragment of $\PA$ obtained by restricting the induction schema
$$
\phi(0) \land \forall x \left(\phi(x) \imp \phi(x + 1)\right) \imp \forall x\, \phi(x), \quad \phi(x) \in \Gamma
$$
to $\Gamma$-formulas \emph{without parameters} is denoted by $\mathsf{I}\Gamma^-$ and called \emph{parameter free $\Gamma$-induction}.
If, in the schema above $\phi(x)$, is allowed to contain parameters, then we obtain the usual $\Gamma$-induction schema and the corresponding 
theory is denoted by $\mathsf{I}\Gamma$. 
We also consider the schema $\mathsf{B}\Sigma_n$ of \emph{$\Sigma_n$-collection}
$$
\forall x < z\, \exists y\, \phi(x, y, a) \imp \exists u\, \forall x < z\, \exists y < u\, \phi(x, y, a), \quad \phi(x, y, a) \in \Sigma_n.
$$
For more details on these theories, see~\cite{KPD}.

All the theories considered in this note are supposed to be \emph{recursively axiomatizable consistent extensions} of $\EA$.
We assume that some standard arithmetization of syntax and the g{\"o}delnumbering of syntactic objects has been fixed.
In particular, we write $\gn{\phi}$ for the (numeral of the) g{\"o}del\-number of $\varphi$.
As usual, each theory $T$ is given to us by an elementary formula $\sigma_T(x)$,
defining the set of axioms of $T$ in the standard model of arithmetic. 
The formula $\sigma_T(x)$ is used in the construction of the formula $\Prf_T(y, x)$ representing the relation ``$y$ codes a $T$-proof of the formula with g{\"o}delnumber $x$''.
The \emph{standard provability predicate} for $T$ is given by $\exists y\, \Prf_T(y, x)$, and we denote this formula by $\Box_T(x)$.
We often write $\Box_T \varphi$ instead of $\Box_T(\gn{\phi})$ and use the notation $\Diamond_T\phi$ for $\neg \Box_T \neg \phi$ .
The sentence $\Diamond_T \top$ is the consistency assertion for $T$ and is also denoted by $\Con(T)$.

The predicate $\Box_T$ satisfies \emph{L{\"o}b's derivability conditions} provably in $\EA$ (cf.~\cite{Bek05en}):
\begin{enumerate}
\item If $T \vdash \varphi$, then $\EA \vdash \Box_T\varphi$.
\item $\EA \vdash \Box_T(\varphi \imp \psi) \imp (\Box_T\varphi \imp \Box_T\psi)$.
\item $\EA \vdash \Box_T\varphi \imp \Box_T\Box_T\varphi$.
\end{enumerate}
Point 3 follows from the general fact known as \emph{provable $\Sigma_1$-completeness}:
$$
\EA \vdash \forall x_1,\dots, x_m \left(\sigma(x_1, \dots, x_m) \imp \Box_T\gn{\sigma(\num{x}_1, \dots, \num{x}_m)}\right),
$$
whenever $\sigma(x_1, \dots, x_m)$ is a $\Sigma_1$-formula. 
Here the underline notation $\ul \varphi(\num{x}) \ur$ stands for the elementarily definable term, representing the elementary function
that maps $k$ to the g{\"o}delnumber of the formula $\ul\varphi(\num{k})\ur$.
In what follows we usually write just $\Box_T \phi(\num{x})$ instead of $\Box_T \gn{\phi(\num{x})}$.

If two theories $T$ and $U$ have the same theorems, we say that they are  \emph{deductively equivalent} and denote this by $T \equiv U$.
If they prove the same arithmetical sentences of complexity $\Gamma$, we write $T \equiv_\Gamma U$.

In this note we are mainly interested in the following three principles (or schemata) for a given arithmetical theory $T$:
\begin{itemize}

\item \emph{local reflection} $\Rfn(T)$: 
$$
\Box_T\phi \imp \phi,
$$
for each arithmetical sentence $\phi$;

\item \emph{uniform reflection} $\RFN(T)$: 
$$
\forall x_1, \dots, x_n\, \left( \Box_T \phi(\num{x}_1, \dots, \num{x}_n) \imp \phi(x_1, \dots, x_n) \right),
$$
for each arithmetical formula $\phi(x_1, \dots, x_n)$;

\item \emph{local $\omega$-consistency} $\omegacon(T)$:
$$
\forall x_1, \dots, x_n\, \Box_T \phi(\num{x}_1, \dots, \num{x}_n) \imp \Diamond_T \forall x_1, \dots, x_n\, \phi(x_1, \dots, x_n),
$$
for each arithmetical formula $\phi(x_1, \dots, x_n)$.

\end{itemize}

For a more detailed analysis of the principles above we also consider their \emph{partial} variants, which are obtained
by imposing the restriction $\phi \in \Gamma$, where $\Gamma$ is some class of arithmetical formulas (usually, $\Sigma_n$ or $\Pi_n$).
Corresponding partial principles are denoted by $\omegacon_\Gamma(T)$, $\Rfn_\Gamma(T)$ and $\RFN_\Gamma(T)$, respectively.
Partial uniform reflection principles satisfy the equivalence $\RFN_{\Pi_{n+1}}(T) \equiv \RFN_{\Sigma_n}(T)$ over $\EA$
for each $n \geqslant 0$ (see \cite[Lemma 2.4]{Bek05en}).

Note that, in the definitions of $\RFN(T)$ and  $\omegacon(T)$ (and their partial analogues), 
by using sequence coding functions, we can, equivalently, restrict these schemata to the formulas $\phi(x)$ with a single free variable $x$.
Let us also mention the following useful principle
known as the \textit{small reflection} (see \cite[Lemma 2.2]{Bek05en}). 
\begin{proposition}\label{prop:smrfn}
For each formula $\phi(x)$
$$
\EA \vdash \forall x, y\, \Box_T \left(\Prf_T(\num{y}, \gn{\phi(\num{x})}) \imp \phi(\num{x})\right).
$$
\end{proposition}
Using this principle one shows that over $\EA$, the schema of uniform reflection $\RFN(T)$ is equivalent to
the schema
$$
\forall x\, \Box_T \phi(\num{x}) \imp \forall x\, \phi(x),
$$
for all arithmetical formulas $\phi(x)$.

In \cite{KayKot} the authors have also considered the following schema $\omegacon^{\mathsf{Th}}(T)$ (note that we use slightly different notation), 
which is a strengthening of $\omegacon(T)$,
\begin{itemize}
\item \emph{local $\omega$-consistency of the theory of the model} $\omegacon^{\mathsf{Th}}(T)$: 
$$
\sigma \land \forall x\, \Box_T( \sigma \imp  \phi(\num{x}))  \imp \Diamond_T (\sigma \land \forall x\, \phi(x)),
$$
for each arithmetical formula $\phi(x)$ with a single free variable $x$ and arithmetical sentence $\sigma$.
\end{itemize} 
It is shown that $\omegacon^{\mathsf{Th}}(\PA)$ implies $\Rfn(\PA)$ (take $\phi(x)$ to be $x = x$) 
and is implied by $\RFN(\PA)$.

It is known that for each $n > 0$ there exists an arithmetical $\Pi_n$-formula 
$\True_{\Pi_n}(x)$ (known as a \emph{truth definition} for $\Pi_n$-formulas) such that
$$
\EA \vdash \forall x_1,\dots, x_m\:(\varphi(x_1, \dots, x_m) \eqv \True_{\Pi_n}(\gn{\varphi(\num{x}_1, \dots, \num{x}_m)})),
$$
for every $\Pi_n$-formula $\varphi(x_1, \dots, x_m)$,
and this fact itself is formalizable in $\EA$, so
$$
\EA \vdash \forall \pi \in \Pi_n\, \Box_T \left(\pi \eqv \True_{\Pi_n}(\pi)\right).
$$

Using these properties, it is not hard to check that over $\EA$, the schema $\RFN_{\Sigma_n}(T)$ is equivalent to its universal instance with $\phi(x)$
taken to be $\neg \True_{\Pi_n}(x)$.

The truth predicates are used to define the so-called strong provability predicates. Namely, the following formula
$$
[n]_T \varphi := \exists \pi \left(\True_{\Pi_n}(\pi) \wedge \Box_T(\pi \imp \varphi)\right),
$$
defines the predicate of \emph{$n$-provability}, i.e., usual provability in $T$ together with all true $\Pi_n$-sentences taken as additional axioms.
The predicate $[n]_T$ satisfies the same derivability conditions as $\Box_T$ and is provably $\Sigma_{n+1}$-complete (see \cite{Bek05en}).
It is easy to see that $\RFN_{\Sigma_{n+1}}(T)$ is equivalent to $\neg [n]_T \bot$ over $\EA$.

The \emph{relativized local reflection principles} $\Rfn^n_\Gamma(T)$ are defined 
analogously to $\Rfn_\Gamma(T)$ but with $n$-provability predicate $[n]_T$ instead of the usual provability $\Box_T$.

\section{Reflection with definable parameters}

In \cite{KayKot} the authors have introduced the following reflection schema 
(we formulate it in a dual form, which is actually used in the proofs there),
which can be seen as an intermediate scheme between local $\Rfn(T)$ and uniform reflection $\RFN(T)$, 
\begin{itemize}
\item \emph{definable reflection} $\DRfn(T)$ 
$$
\forall x \left(\phi(x) \land \forall y\, \left(\phi(y) \imp y = x\right) \imp 
\Diamond_T \left[\phi(\num{x}) \land \forall y\, \left(\phi(y) \imp y = \num{x} \right)\right] \right),
$$
for each arithmetical formula $\phi(x)$ with a single free variable $x$.
\end{itemize}
It was shown that $\DRfn(\PA)$ is actually equivalent to $\RFN(\PA)$ using indicators.
Let us include a proof of this fact without using indicators. 
\begin{proposition}
Over $\PA$, $\DRfn(T) \equiv \RFN(T)$.
\end{proposition} 
\begin{proof}
Clearly, $\EA + \RFN(T) \vdash \DRfn(T)$. 
To prove the converse fix a formula $\phi(x)$, and let
$\delta(x)$ be the formula $\neg \phi(x) \land \forall z < x\, \phi(z)$. We have
\begin{align*}
\PA + \DRfn(T) \vdash \exists x\, \neg \phi(x) 
&\imp \exists x\, \left(\delta(x) \land \forall y\, (\delta(y) \imp y = x)\right)\\
&\imp \exists x\, \Diamond_T \left(\delta(\num{x}) \land \forall y\, (\delta(y) \imp y = \num{x})\right)\\
&\imp \exists x\, \Diamond_T \neg \phi(\num{x}),
\end{align*}
where the first implication uses induction and the second one uses $\DRfn(T)$. By contraposition we obtain
$$
\PA + \DRfn(T) \vdash \forall x\, \Box_T \phi(\num{x}) \imp \forall x\, \phi(x),
$$
that is, $\PA + \DRfn(T) \vdash \RFN(T)$.
\end{proof}

In this section we consider a refined version of the definable reflection principle, namely, \emph{uniform reflection with $\Sigma_n$-definable parameters},
and use it to give a model-theoretic proof of the $\Sigma_{n+2}$-conservativity of uniform reflection over relativized local reflection (Theorem \ref{th:s2conserv}).
We also discuss a relationship between these principles and the schemata of induction with definable parameters introduced
by A.~Cord{\'o}n-Franco et al. in \cite{CordFr}.

Recall the uniform $\Sigma_k$-reflection principle $\RFN_{\Sigma_k}(T)$:
$$
\forall x\left(\Box_T \phi(\num{x}) \imp \phi(x)\right), 
$$
where $\phi(x)$ is a $\Sigma_k$-formula. 
If we require the variable $x$ above to range only over the standard elements (numerals), then we get the schema that is equivalent to 
local reflection $\Rfn_{\Sigma_k}(T)$. 
We investigate the question: can we expand the range of $x$ to some nonstandard elements while still obtaining the equivalent schema?

Formally, we define the following schema of
\begin{itemize}
\item \emph{uniform $\Sigma_k$-reflection with $\Sigma_n$-definable parameters} $\RFN^{K^n}_{\Sigma_k}(T)$:
$$
\forall x\,
\left(\Def_\delta(x) \imp \left(\Box_T\phi(\num{x}) \imp \phi(x)
\right)
\right),
$$
for each $\Sigma_n$-formula  $\delta(x)$ and $\Sigma_k$-formula $\phi(x)$,
where 
$$
\Def_\delta(x) :=
\delta(x) \land \forall y, z\, (\delta(y) \land \delta(z) \imp y = z)
$$
is the formula asserting that $x$ is the unique element satisfying $\delta(x)$.
\end{itemize}
We aim at proving that these reflection principles are equivalent to their local counterparts.

\begin{proposition}\label{prop:defrfn}
For each $k > n \geqslant 0$ we have
$\EA + \Rfn^n_{\Sigma_k}(T) \vdash \RFN^{K^{n+1}}_{\Sigma_k}(T)$.
\end{proposition}
\begin{proof}
Fix some $\Sigma_{n+1}$-formula $\delta(x)$ and $\Sigma_k$-formula $\phi(x)$. We will derive the corresponding axiom
of $\RFN^{K^{n+1}}_{\Sigma_k}(T)$.
Using provable $\Sigma_{n+1}$-completeness and an instance of $\Rfn^n_{\Sigma_k}(T)$ for the $\Sigma_k$-sentence 
$\exists u\,(\delta(u) \land \phi(u))$ (we use $n < k$ here)
we derive
\begin{align*}
\EA + \Rfn^n_{\Sigma_k}(T) \vdash \Def_\delta(x) \land \Box_T\phi(\num{x}) & \imp [n]_T\delta(\num{x}) \land \Box_T\phi(\num{x})\\
& \imp [n]_T (\delta(\num{x}) \land \phi(\num{x}))\\
& \imp [n]_T \exists u\,(\delta(u) \land \phi(u))\\
& \imp \exists u\,(\delta(u) \land \phi(u)).
\end{align*}
Since $\EA \vdash \Def_\delta(x) \land \delta(u) \imp x = u$, we get
\begin{align*}
\EA + \Rfn^n_{\Sigma_k}(T) \vdash \Def_\delta(x) \land \delta(u) \land \phi(u) & \imp
(x = u) \land \phi(u) \\
& \imp \phi(x),
\end{align*}
whence $\EA + \Rfn^n_{\Sigma_k}(T) \vdash \left(\Def_\delta(x)\land \exists u\,(\delta(u) \land \phi(u))\right) \imp \phi(x)$, and
so 
$$
\EA + \Rfn^n_{\Sigma_k}(T) \vdash \Def_\delta(x) \imp \left(\Box_T\phi(\num{x})  \imp \phi(x)\right),
$$ 
as required.
\end{proof}


In particular, the local $\Sigma_k$-reflection is equivalent to the uniform $\Sigma_k$-reflection with $\Sigma_1$-definable parameters.
\begin{corollary}\label{col:rfndef0}
For each $k > 0$ over $\EA$, $\Rfn_{\Sigma_k}(T) \equiv \RFN^{K^1}_{\Sigma_k}(T)$.
\end{corollary}
\begin{proof}
By Proposition \ref{prop:defrfn} it suffices to show that 
$\EA + \RFN^{K^1}_{\Sigma_k}(T) \vdash \Box_T\phi \imp \phi$, where $\phi$ is a $\Sigma_k$-sentence. Define $\delta(x)$ to be the formula $(x = \num{0})$ 
and consider the corresponding axiom of $\RFN^{K^1}_{\Sigma_k}(T)$. 
Note that $\EA \vdash \Def_\delta(\num{0})$,
whence, by instantiating the axiom with $\num{0}$, we obtain $\Box_T \phi \imp \phi$.
\end{proof}

Let us recall several notions related to the models of arithmetic (for more details, see \cite{KPD}).
Given a model $M$ and a natural number $n$ we denote by $K^n(M)$ the substructure of $M$ consisting 
of all $\Sigma_n$-definable elements without parameters. Given a substructure $N \subseteq M$, we say
that $N$ is a \emph{$\Sigma_n$-elementary substructure} (denoted $N \prec_{\Sigma_n} M$)
if and only if for each $\Sigma_n$-formula $\sigma(x_1, \dots, x_m)$ and $a_1, \dots a_m \in N$
$$
N \models \sigma(a_1, \dots, a_m) \Longleftrightarrow M \models \sigma(a_1,\dots,a_m).
$$

The models $K^{n+1}(M)$ possess the following useful property (see Remark (i) after Theorem 2.1 in \cite{KPD}).
\begin{lemma}\label{lm:defmod}
For each $n \geqslant 0$, $K^{n+1}(M) \prec_{\Sigma_{n+1}} M$,
whenever  $M \models \ISi^-_n$.
\end{lemma}

We have the following generalization of Corollary \ref{col:rfndef0}.
\begin{proposition}\label{prop:rfndef}
For each $k > n \geqslant 0$ over
$\EA + \ISi^-_n$, $\Rfn^n_{\Sigma_k}(T) \equiv \RFN^{K^{n+1}}_{\Sigma_k}(T)$.
\end{proposition}

\begin{proof}
By Proposition \ref{prop:defrfn} $\EA + \Rfn^n_{\Sigma_k}(T) \vdash \RFN^{K^{n+1}}_{\Sigma_k}(T)$,
so we only prove the converse implication (over $\EA + \ISi^-_n$). Fix some $\Sigma_k$-sentence $\sigma$ and
a model $M$ of $\EA + \ISi^-_n + \RFN^{K^{n+1}}_{\Sigma_k}(T)$. 
We will show that $M \models [n]_T\sigma \imp \sigma$, whence the result follows,
so assume $M \models [n]_T\sigma$.
By Lemma \ref{lm:defmod} we have $K^{n+1}(M) \prec_{\Sigma_{n+1}} M$, implying
$K^{n+1}(M) \models [n]_T\sigma$, since $[n]_T\sigma$ is a 
$\Sigma_{n+1}$-sentence. It follows that there exists a $\pi \in K^{n+1}(M)$
such that
$$
K^{n+1}(M) \models \True_{\Pi_n}(\pi) \land \Box_T(\pi \imp \sigma).
$$
Using $K^{n+1}(M) \prec_{\Sigma_{n+1}} M$ again, we get
$M \models  \True_{\Pi_n}(\pi) \land \Box_T(\pi \imp \sigma)$,
whence 
\begin{equation}\label{eq:2}
M \models  \True_{\Pi_n}(\pi) \land \Box_T(\True_{\Pi_n}(\pi) \imp \sigma),
\end{equation}
since $\EA \vdash \forall \pi \in \Pi_n\, \Box_T \left(\pi \eqv \True_{\Pi_n}(\pi)\right)$ and $M \models \EA$.
Let $\delta(x)$ be a $\Sigma_{n+1}$-formula such that $M \models \Def_\delta(\pi)$. We have
$$
M \models \Def_\delta(\pi) \land \Box_T(\True_{\Pi_n}(\pi) \imp \sigma).
$$
Now, since $\True_{\Pi_n}(\pi) \imp \sigma$ is a $\Sigma_k$-formula, we can use $\RFN^{K^{n+1}}_{\Sigma_k}(T)$ in $M$ to obtain
$$
M \models \True_{\Pi_n}(\pi) \imp \sigma,
$$
whence $M \models \sigma$, as required, since $M \models \True_{\Pi_n}(\pi)$  by \eqref{eq:2}.
\end{proof}

Now we can give a relatively short model-theoretic proof of the following consequence 
of the so-called \emph{reduction property} (see \cite[Proposition 4.6]{Bek03}).

\begin{theorem}\label{th:s2conserv}
If $U$ is a $\Pi_{n+2}$-axiomatized extension of $\EA$, then
$$
U + \RFN_{\Sigma_{n+1}}(T) \equiv_{\Sigma_{n+2}} U + \Rfn^n_{\Sigma_{n+1}}(T).
$$ 
\end{theorem}
\begin{proof}
The inclusion $\supseteq$ is clear from the definitions of the schemata and $[n]_T$, so we only prove the converse.
Assume $U  + \RFN_{\Sigma_{n+1}}(T) \vdash \sigma$, where $\sigma$ is a $\Sigma_{n+2}$-sentence.
Fix an arbitrary model $M \models U + \Rfn^n_{\Sigma_{n+1}}(T)$. 
We will show that $M \models \sigma$.
By point (i) of Theorem 1 from \cite{Bek99}, $\EA +  \Rfn^n_{\Sigma_{n+1}}(\EA) \vdash \ISi^-_n$, 
whence, certainly, $U + \Rfn^n_{\Sigma_{n+1}}(T) \vdash \ISi^-_n$, 
and so, in particular, $M \models \EA + \ISi^-_n$.
Lemma \ref{lm:defmod} then implies $K^{n+1}(M) \prec_{\Sigma_{n+1}} M$. 

We have $K^{n+1}(M) \models U$, since $M \models U$ and $U$ is a $\Pi_{n+2}$-extension 
(the truth of $\Pi_{n+2}$-sentences is preserved downwards by the relation $\prec_{\Sigma_{n+1}}$). 
We will show that $K^{n+1}(M) \models \RFN_{\Sigma_{n+1}}(T)$. In this case $K^{n+1}(M) \models U + \RFN_{\Sigma_{n+1}}(T)$, whence, by the assumption,
$K^{n+1}(M) \models \sigma$, so $M \models \sigma$, as required, since
$\sigma$ is a $\Sigma_{n+2}$-sentence (the truth of $\Sigma_{n+2}$-sentences is preserved upwards by the relation $\prec_{\Sigma_{n+1}}$).

The rest of the proof is close to that of Proposition \ref{prop:rfndef}.
Aiming for a contradiction, assume $K^{n+1}(M) \not \models \RFN_{\Sigma_{n+1}}(T)$, i.e.,
there exists a $\Sigma_{n+1}$-formula $\phi(x)$ and an element $a \in K^{n+1}(M)$ with
$
K^{n+1}(M) \models \Box_T \phi(\num{a}) \land \neg \phi(a),
$
which, using $K^{n+1}(M) \prec_{\Sigma_{n+1}} M$, implies
\begin{equation}\label{eq:3}
M \models \Box_T \phi(\num{a}) \land \neg \phi(a).
\end{equation}
Let $\delta(x)$ be a $\Sigma_{n+1}$-formula such that $M \models \Def_\delta(a)$. 
Since $M \models \EA + \Rfn^n_{\Sigma_{n+1}}(T)$, Proposition \ref{prop:defrfn} implies that 
$M \models \RFN^{K^{n+1}}_{\Sigma_{n+1}}(T)$ and, in particular,
$$
M \models \Def_\delta(a) \imp (\Box_T\phi(\num{a}) \imp \phi(a)),
$$
since $\phi(x)$ is a $\Sigma_{n+1}$-formula. Together with \eqref{eq:3} this yields 
$M \models \phi(a)$, which contradicts $M \models \neg \phi(a)$ in \eqref{eq:3}.
\end{proof}

Let us also note the following relationship between Proposition \ref{prop:rfndef} and the following proposition proved in \cite[Proposition 4.1]{CordFr}
(we use slightly different notation).
\begin{proposition}\label{prop:cordfr}
For each $n \geqslant 0$ over $\EA + \ISi^-_n$, 
$$
\IPi^-_{n+1} \equiv \mathsf{I}(\Sigma^-_{n+1}, K^{n+1}) \equiv \mathsf{I}(\Sigma_{n+1}, I^{n+1}, K^{n+1}_1).
$$
\end{proposition}
Here $\mathsf{I}(\Sigma^-_{n+1}, K^{n+1})$ is the local variant of $\Sigma_{n+1}$-induction schema, where the conclusion of the
induction axiom is relativized to $\Sigma_{n+1}$-definable elements 
(for the formal definitions of this schema and $\mathsf{I}(\Sigma_{n+1}, I^{n+1}, K^{n+1}_1)$, see \cite{CordFr}).
Thus, Proposition \ref{prop:rfndef} can be seen as an analogue of Proposition \ref{prop:cordfr} but
for the reflection principles instead of the induction schemata.

The connection between the two propositions is based on the following fact
(see \cite[Theorem 1 (ii)]{Bek99} for $n > 0$ and \cite[Theorem 5.2]{CordFr} for $n = 0$),
for each $n \geqslant 0$
$$
\IPi^-_{n+1} \equiv \EA + \Rfn^n_{\Sigma_{n+2}}(\EA).
$$
In view of this result, Propositions \ref{prop:rfndef} and \ref{prop:cordfr} we have the following 
\begin{corollary}\label{col:cordfr}
For each $n \geqslant 0$ over $\EA + \ISi^-_n$,
$$
\RFN^{K^{n+1}}_{\Sigma_{n+2}}(\EA) \equiv \mathsf{I}(\Sigma^-_{n+1}, K^{n+1}) \equiv \mathsf{I}(\Sigma_{n+1}, I^{n+1}, K^{n+1}_1).
$$
\end{corollary}

This may be contrasted with the famous result by D. Leivant \cite{Lei83} and H. Ono \cite{Ono87} (cf. \cite[Theorem 7]{Bek05en}),
that $\EA + \RFN_{\Sigma_{n+2}}(\EA) \equiv \ISi_{n+1}$ for each $n \geqslant 0$, and a result by L.~Beklemishev (see \cite[Theorem 1 (i)]{Bek99}), 
that $\EA + \Rfn^{n+1}_{\Sigma_{n+2}}(\EA) \equiv \ISi^-_{n+1}$ for each $n \geqslant 0$.
In our case we also have the equivalence between certain forms of $\Sigma_{n+2}$-reflection for $\EA$ and $\Sigma_{n+1}$-induction,
namely, for the versions of reflection and induction restricted to $\Sigma_{n+1}$-definable elements.

\section{Local reflection, $\omega$-consistency and 1-provability}

In this section we show how Smory{\'n}ski's characterization of the local $\omega$-consistency assertions in terms of reflection principles
can be used to give a quick solution to the problems from \cite{KayKot} mentioned in the introduction.
We also prove that the schema $\omegacon^{\mathsf{Th}}(T)$, introduced in \cite{KayKot}, is actually equivalent to $\RFN(T)$ (Theorem \ref{th:oconth}). 
In addition, we study the proof-theoretic strength of 
Feferman's theorem and its generalized versions (Theorem \ref{th:rfnsi2}). 
In particular, we obtain an alternative characterization of $\ISi_1$ in terms of Feferman's theorem for $\EA$ (Corollary \ref{cor:isi1rfn}).

The following characterization of the local $\omega$-consistency schemata in terms of reflection principles 
was obtained by C.~Smory{\'n}ski (see \cite[Theorem 1.1.c]{Smo77a}).
\begin{theorem}\label{th:ocon-s2rfn}
Over $\EA$,
\begin{enumerate}

\item[(i)] $\omegacon_{\Sigma_n}(T) \equiv \Rfn_{\Sigma_2}(T + \RFN_{\Sigma_n}(T))$ for each $n > 0$,
\item[(ii)] $\omegacon_{\Sigma_0}(T) \equiv \Rfn_{\Sigma_1}(T)$.

\end{enumerate}

\end{theorem}

The following theorem, which solves Problem 8.4, is a consequence of Theorem~\ref{th:ocon-s2rfn}.
\begin{theorem}\label{th:oconrcon}
$\EA + \omegacon_{\Sigma_1}(T) \vdash \Con(T + \Rfn(T)).$
\end{theorem}
\begin{proof}
Point (i) of Theorem \ref{th:ocon-s2rfn} implies that
$\EA + \omegacon_{\Sigma_1}(T) \vdash \Con(T + \RFN_{\Sigma_1}(T))$,
whence the result follows, once we show that
\begin{equation}\label{eq:aim}
\EA \vdash \Con(T + \RFN_{\Sigma_1}(T)) \imp  \Con(T + \Rfn(T)).
\end{equation}
We clearly have 
$\EA \vdash \Con(T + \RFN_{\Sigma_1}(T)) \imp \Con(T + \Rfn_{\Sigma_1}(T))$.
Also the following $\Pi_1$-conservativity result is provable in $\EA$ (see \cite[Proposition 5.1.(ii)]{Bek03}),
$$
T + \Rfn_{\Sigma_1}(T) \equiv_{\Pi_1} T + \Rfn(T).
$$
As a consequence, these two theories are provably equiconsistent
$$
\EA \vdash  \Con(T + \Rfn_{\Sigma_1}(T)) \eqv \Con(T + \Rfn(T)),
$$
whence \eqref{eq:aim} follows immediately.

\end{proof}

In particular, $\PA + \omegacon_{\Sigma_1}(\PA) \vdash \Con(\PA + \Rfn(\PA))$, which solves Problem 8.4 
and, consequently, Problems 8.1 and 8.3 as well, because then 
$$
\PA + \Rfn(\PA) \not \vdash \omegacon(\PA),
$$ 
by G{\"o}del's second incompleteness theorem, whence, certainly,
$$
\PA + \Rfn(\PA) \not \vdash  \omegacon^{\mathsf{Th}}(\PA).
$$

Recall from Section 2 that $\omegacon^{\mathsf{Th}}(\PA)$ is implied by $\RFN(\PA)$. We prove that these 
two schemata are equivalent for an arbitrary theory $T$, which provides an alternative solution to Problem 8.3.
\begin{theorem}\label{th:oconth}
Over $\EA$, $\omegacon^{\mathsf{Th}}(T) \equiv \RFN(T)$.
\end{theorem}
\begin{proof}

To derive $\omegacon^{\mathsf{Th}}(T)$ from $\RFN(T)$ argue as follows
\begin{align*}
\EA + \RFN(T) \vdash \sigma \land  \forall x\, \Box_T( \sigma \imp  \phi(\num{x}))  
&\imp \sigma \land \forall x\, (\sigma \imp \phi(x))\\
&\imp \sigma \land \forall x\, \phi(x)\\
&\imp \Diamond_T (\sigma \land \forall x\, \phi(x)),
\end{align*}
where the first and the last implications use $\RFN(T)$.

To show the converse fix a formula $\phi(x)$ with a single free variable $x$. 
We show
\begin{equation}\label{eq:oconth1}
\EA + \omegacon^{\mathsf{Th}}(T) \vdash  \forall x\, \Box_T \phi(\num{x}) \imp \forall x\, \phi(x),
\end{equation}
whence $\EA + \omegacon^{\mathsf{Th}}(T) \vdash \RFN(T)$ follows. 
Denote the sentence $\neg \forall x\, \phi(x)$ by $\sigma$
and argue by contraposition
\begin{align*}
\EA + \omegacon^{\mathsf{Th}}(T) \vdash  \sigma \land \forall x\, \Box_T \phi(\num{x}) 
&\imp \sigma \land \forall x\, \Box_T (\sigma \imp \phi(\num{x}))\\
&\imp \Diamond_T(\sigma \land \forall x\, \phi(x))\\
&\imp \Diamond_T \bot\\
&\imp \bot,
\end{align*}
which yields \eqref{eq:oconth1}.

\end{proof}

In particular, $\PA + \Rfn(\PA) \not \vdash  \omegacon^{\mathsf{Th}}(\PA)$, since $\PA + \Rfn(\PA) \not \vdash \RFN(\PA)$,
which gives another solution to Problem 8.3.

In order to formulate these corollaries of Theorem \ref{th:ocon-s2rfn} for an arbitrary theory $T$ in place of $\PA$ we recall the following definition.
A theory $T$ is said to have \emph{infinite characteristic}, if the theory $T_\omega$ is consistent, where
$$
T_0 := T, \quad T_{n+1} := T_n + \Con(T_n), \quad T_\omega := \bigcup_{n < \omega} T_n.
$$

It is known that the theory $T + \Rfn(T)$ is equiconsistent with $T_\omega$ (see \cite[Corollary 2.35]{Bek05en}) .
This fact together with Theorem \ref{th:oconrcon} and G{\"o}del's second incompleteness theorem yields the following
\begin{corollary}
For each theory $T$ of infinite characteristic 
$$
T + \Rfn(T) \not \vdash \omegacon_{\Sigma_1}(T).
$$
\end{corollary}

\noindent {\bf Remark.}
The author learned (personal communication) that the mentioned problems were already solved by V.~Shavrukov back in 2010 (unpublished). 
His proof, which he kindly provided to us, also uses Smory{\'n}ski's characterization of $\omegacon(T)$ and is almost identical to ours except one minor point. 
Let us present his proof and make some comments concerning its relationship to our proof. It goes as follows
\begin{align*}
\PA \vdash \omegacon_{\Sigma_1}(\PA) &\imp \Rfn_{\Sigma_2}(\PA + \RFN_{\Pi_2}(\PA))\\
&\imp  \Rfn_{\Sigma_2}(\PA + \RFN_{\Pi_2}(\PA + \Rfn(\PA)))\\
&\imp \Con(\PA + \Con(\PA + \Rfn(\PA)))\\
&\imp \Con(\PA + \Rfn(\PA)),
\end{align*} 
which again solves Problem 8.4 and the other two problems as well. 
The second implication in the above proof uses the following fact \cite[Theorem 5.1.i]{Smo77a},
$$
\PA \vdash  \RFN_{\Pi_2}(\PA) \eqv  \RFN_{\Pi_2}(\PA + \Rfn(\PA)),
$$
While this is true for the theories containing $\ISi_1$, it does not hold for weaker theories, e.g., $\EA$ or $\EA^+$.
\bigskip

Let us comment on this issue a little bit more.
It is well-known (the so-called Feferman theorem  \cite[Theorem 4.5]{Fef62}) that the local reflection schema $\Rfn(S)$ 
for a consitent r.e.~theory $S$ is contained
in $S$ together with the set $\mathsf{Th}_{\Pi_1}(\nat)$ of all true $\Pi_1$-sentences.

The usual informal proof goes as follows. For a given instance $\Box_S\psi \imp \psi$ of $\Rfn(S)$ consider two possibilities:
either $S \vdash \psi$ or $S \not \vdash \psi$. In the first case we have $S \vdash \Box_S\psi \imp \psi$,
and in the second case, since $\neg \Box_S \psi$ is a true $\Pi_1$-sentence, 
we obtain $S + \mathsf{Th}_{\Pi_1}(\nat) \vdash \neg \Box_S \psi$, implying 
$S + \mathsf{Th}_{\Pi_1}(\nat) \vdash \Box_S\psi \imp \psi$.

It follows that each theorem of the theory $S + \Rfn(S)$ is provable in $S + \mathsf{Th}_{\Pi_1}(\nat)$.
Natural formalization of the above proof gives
\begin{equation}\label{eq:fefax}
\EA \vdash \forall \psi\, [1]_S (\Box_S\psi \imp \psi).
\end{equation}
However, the formalization of 1-provability in $S$ of each theorem of $S + \Rfn(S)$, i.e., the following $\Pi_3$-sentence
\begin{equation}\label{eq:fefthm}
\forall \phi \left(\Box_{S + \Rfn(S)}\phi \imp [1]_S\phi\right),
\end{equation}
cannot be proven in $\EA$ (and even in $\EA^+$). 

To derive \eqref{eq:fefthm} in $\ISi_1$ from 1-provability of each axiom \eqref{eq:fefax},
one can use $\Sigma_2$-collection schema $\BSi_2$ (since the provability predicate $[1]_S$ is $\Sigma_2$) 
in the same way as $\BSi_1$ is usually applied to conclude the provability of each theorem given 
that of each axiom for $\Sigma_1$-provability predicates (see, e.g., \cite[Proposition 5.1]{Bek03} for this type of argument). 
After that $\Pi_3$-conservativity of $\BSi_2$ over $\ISi_1$ (see \cite[Theorem 4.1]{Bek98}) 
yields the provability of \eqref{eq:fefthm} in $\ISi_1$ (cf. Theorem \ref{th:rfnsi2} below for $T = \EA$).
As for the non-provability of \eqref{eq:fefthm} in $\EA^+$, we have the following

\begin{proposition}
$\EA^+ \not \vdash \forall \phi \left(\Box_{S + \Rfn(S)}\phi \imp [1]_S\phi\right)$.
\end{proposition}
\begin{proof}
Aiming at a contradiction assume $\EA^+$ proves \eqref{eq:fefthm}. Instantiating $\phi$ with $\bot$, 
we get that $\Con(S + \Rfn(S))$ is provable in $\EA^+ + \neg [1]_S \bot$, i.e., 
in $\EA + \RFN_{\Sigma_1}(S)$, since $\EA^+$ is contained in the latter theory. 
But $\EA + \RFN_{\Sigma_1}(S)$ is $\Sigma_2$-conservative over $\EA + \Rfn_{\Sigma_1}(S)$ (see \cite[Proposition 4.6]{Bek03}),
whence 
$$
\EA + \Rfn_{\Sigma_1}(S) \vdash \Con(S + \Rfn(S)),
$$ 
contradicting G{\"o}del's second incompleteness theorem.
\end{proof}

A natural question to ask then is whether the formalized version of Feferman's result \eqref{eq:fefthm} is actually equivalent to $\ISi_1$
over a weaker theory? We establish this equivalence and, in fact, we obtain a more general result using 
$\EA^+$-provable characterization of the set $C_S(T) = \{\phi \mid T \vdash [1]_S\phi\}$
in terms of iterated local reflection principles for $S$.

See \cite{KolmBek} for the precise definitions of iterated local reflection principles $\Rfn(S)_\alpha$ 
and the corresponding notion of $\Sigma^0_2$-ordinal $|T|_{\Sigma^0_2}$ of a theory.
The following theorem, which is a formalization of the equivalence $C_S(T) \equiv \Rfn(S)_{1 + \alpha}$ 
for ``nice'' ($\EA^+$-provably $\Sigma^0_2$-regular) theories $T$, was obtained in \cite[Theorem 5.3]{KolmBek}. 
\begin{theorem}\label{th:cstrfn}
For each $\EA^+$-provably $\Sigma^0_2$-regular theory $T$ with $\Sigma^0_2$-ordinal $\alpha$
$$
\EA^+ \vdash \forall \phi \left(\Box_T[1]_S\phi \eqv \Box_{\Rfn(S)_{1+\alpha}}\phi\right).
$$
\end{theorem}
Natural examples of such theories $T$ and corresponding $\Sigma^0_2$-ordinals $|T|_{\Sigma^0_2}$ are
listed in the table below. 
Here $\omega_0 := 1$, $\omega_{n+1} := \omega^{\omega_n}$ and $\varepsilon_0 = \sup\{\omega_n \mid n < \omega\}$. See \cite{KolmBek} for more examples and details.

\begin{center}
  \begin{tabular}{ | c | c | c | c | c | c | }
    \hline
    $T$ & $\EA$ & $\ISi_1$ & $\ISi_n$ & $\PA$ & $\Rfn(\EA)_\beta$ \\ \hline
    $|T|_{\Sigma^0_2}$ & $0$ & $\omega$ & $\omega_n$ & $\varepsilon_0$ &  $\beta$ \\ \hline
  \end{tabular}
\end{center}

Let us denote by $\RFN_{[1]_S}(T)$ the uniform $T$-reflection schema for $[1]_S$-formulas,
i.e., the following $\Pi_3$-sentence
$$
\forall \phi \left(\Box_T[1]_S\phi \imp [1]_S\phi\right).
$$
Theorem \ref{th:cstrfn} immediately yields the following
\begin{corollary}\label{col:cstrfn}
For each $\EA^+$-provably $\Sigma^0_2$-regular theory $T$ with $\Sigma^0_2$-ordinal $\alpha$,
the following are equivalent over $\EA^+$
\begin{itemize}
\item[(i)] $\RFN_{[1]_S}(T)$,
\item[(ii)] $\forall \phi \left(\Box_{\Rfn(S)_{1+\alpha}}\phi \imp [1]_S\phi\right)$.
\end{itemize}
\end{corollary}

Now, we need to relate $\RFN_{[1]_S}(T)$ with the schema $\RFN_{\Sigma_2}(T)$. 
This is done by applying the generalized version of the Friedman-Goldfarb-Harrington principle (see \cite{Joo} for various generalizations of this principle).
Note the presence of $\BSi_1$ in our formulation, since our formalization of 1-provability predicate $[1]_S$ is slightly different from that of \cite{Joo}.

\begin{lemma}\label{lm:fgh}
For each $\Sigma_2$-formula $\sigma(x)$ there is a formula $\phi(x)$ such that
\begin{itemize}
\item[(i)] $\EA \vdash \sigma(x) \lor [1]_S \bot \imp [1]_S \phi(\num{x})$,
\item[(ii)] $\EA + \BSi_1 \vdash [1]_S \phi(\num{x}) \imp \sigma(x) \lor [1]_S\bot$. 
\end{itemize}
\end{lemma}
\begin{proof}
Let $\sigma(x)$ be $\exists y\, \psi(y, x)$ for $\Pi_1$-formula $\psi(y, x)$, and 
let $\chi(z, \theta)$ be the formula $\exists \pi < z \left(\True_{\Pi_1}(\pi) \land \Prf_S(z, \gn{\pi \imp \theta})\right)$.
We clearly have $\EA \vdash [1]_S\theta \eqv \exists z\, \chi(z, \theta).$
Consider the following fixed-point
$$
\EA \vdash \phi(x) \eqv \exists y \left(\psi(y, x) \land \forall z < y\, \neg \chi(z, \gn{\phi(\num{x})})\right).
$$
To prove $(i)$ note that, clearly, $[1]_S\bot$ implies $[1]_S \phi(\num{x})$, so we only show
$$
\EA \vdash \sigma(x) \imp [1]_S \phi(\num{x}).
$$
Fix $y$ such that $\psi(y, x)$ and argue under $[1]_S$. The proof of $\phi(\num{x})$ goes by considering two cases. 
If the second member of the conjunction in the definition of $\phi(x)$
is true, i.e., $\forall z < \num{y}\, \neg \chi(z, \gn{\phi(\num{x})})$, then, since we also have $\psi(\num{y}, \num{x})$, being 
a true $\Pi_1$-sentence, we obtain $\phi(\num{x})$. Otherwise, $\exists z \leqslant \num{y}\, \chi(z, \gn{\phi(\num{x})})$. 
This formula is equivalent to the disjunction, since $\num{y}$ is a numeral, so it is sufficient to prove the following
$$
\EA \vdash \forall y\, \Box_S \left(\bigvee_{z \leqslant y} \chi(\num{z}, \gn{\phi(\num{x})}) \imp \phi(\num{x}) \right). 
$$
The proof of the above goes by induction on $y$, which is formalizable in $\EA$, since we can bound the code
of the corresponding $S$-proof by some elementary term $p(y)$. 
The induction step goes through, once the following fact is established
$$
\EA \vdash \forall z\, \Box_S \left(\chi(\num{z}, \gn{\phi(\num{x})}) \imp \phi(\num{x})\right).
$$
This follows from the small reflection principle (see Proposition \ref{prop:smrfn}).
Indeed, argue under $\Box_S$ and assume $\chi(\num{z}, \gn{\phi(\num{x})})$, that is,
$\exists \pi < \num{z} \left(\True_{\Pi_1}(\pi) \land \Prf_S(\num{z}, \gn{\pi \imp \phi(\num{x})})\right)$.
Again, this $\Pi_1$-sentence $\pi$ is standard, so $\True_{\Pi_1}(\pi)$ is equivalent to $\pi$ 
(and the proof of this fact is bounded by the fixed elementary term of $z$ and, consequently, of $y$), since we have
$$
\EA \vdash \forall \pi \in \Pi_1\, \Box_S\left(\True_{\Pi_1}(\pi) \eqv \pi\right). 
$$
The small reflection principle yields $\pi \imp \phi(\num{x})$ from  $\Prf_S(\num{z}, \gn{\pi \imp \phi(\num{x})})$, whence $\phi(\num{x})$ follows.

To prove $(ii)$ assume $[1]_S\phi(\num{x})$ and $\neg \sigma(x)$. We prove $[1]_S\neg\phi(\num{x})$, whence $[1]_S\bot$ follows. 
Recall that 
$$
\EA \vdash \neg \phi(x) \eqv \forall y \left(\psi(y, x) \imp \exists z \leqslant y\, \chi(z, \gn{\phi(\num{x})})\right).
$$
Denote the formula in the brackets by $\theta(y, x)$. 
Fix $u$ such that $\chi(u, \gn{\phi(\num{x})})$ (follows from $[1]_S\phi(\num{x})$) and argue under $[1]_S$. 
To prove $\neg\phi(\num{x})$, i.e., $\forall y\, \theta(y, \num{x})$, we fix an arbitrary $y$ and consider two cases. 
If $\num{u} \leqslant y$, then since we have $\chi(\num{u}, \gn{\phi(\num{x})})$ under $[1]_S$, being a true $\Sigma_2$-sentence, 
we obtain $\exists z \leqslant y\, \chi(z, \gn{\phi(\num{x})})$, where we can take $z$ to be $\num{u}$, which itself implies $\theta(y, \num{x})$.
If we have $y < \num{u}$, then it is left to prove the following ($u$ can be treated as a free variable)
$$
\EA + \BSi_1 \vdash \neg \sigma(x) \imp [1]_S \forall y < \num{u}\, \neg \psi(y, \num{x}),
$$
since the formula under $[1]_S$ implies $\theta(y, \num{x})$. Again, since $\num{u}$ is standard, it is sufficient to prove
$$
\EA + \BSi_1 \vdash \neg \sigma(x) \imp \Box_S \bigwedge_{y < u} \neg \psi(\num{y}, \num{x}).
$$  
By provable $\Sigma_1$-completeness we have
$$
\EA \vdash \neg \sigma(x) \imp \forall y < u\, \Box_S \neg \psi(\num{y}, \num{x}),
$$
so we need to prove 
$$
\EA + \BSi_1 \vdash \forall y < u\, \Box_S \neg \psi(\num{y}, \num{x}) \imp \Box_S \bigwedge_{y < u} \neg \psi(\num{y}, \num{x}).
$$ 
This is where $\Sigma_1$-collection schema $\BSi_1$ is used. Applying $\Sigma_1$-collection to the formula
$\forall y < u\, \Box_S \neg \psi(\num{y}, \num{x})$, we get a sequence of $S$-proofs $\psi(\num{y}, \num{x})$ for $y < u$,
and then we concatenate them (and make some other trivial transformations) to obtain the proof of $\bigwedge_{y < u} \neg \psi(\num{y}, \num{x})$.

Combining the above proofs of $\theta(y, \num{x})$ under $[1]_S$ for the two cases considered above, we obtain
$$
\EA + \BSi_1 \vdash \neg \sigma(x) \land [1]_S\phi(\num{x}) \imp [1]_S\neg \phi(\num{x}),
$$
which proves $(ii)$.
\end{proof}

The next lemma shows that $\RFN_{[1]_S}(T)$ is equivalent to the full $\Sigma_2$-reflection schema $\RFN_{\Sigma_2}(T)$
over  $\EA + \BSi_1 + \neg [1]_S\bot$.

\begin{lemma}\label{lm:rfnsig2}
\begin{itemize}
\item[(i)] $\EA \vdash \RFN_{\Sigma_2}(T) \imp \RFN_{[1]_S}(T)$,
\item[(ii)] $\EA + \BSi_1 + \neg [1]_S\bot \vdash \RFN_{[1]_S}(T) \imp \RFN_{\Sigma_2}(T)$.
\end{itemize}

\end{lemma}
\begin{proof}
Point $(i)$ is clear, since $[1]_S\phi$ is a $\Sigma_2$-formula.

The reverse implication $(ii)$ follows from Lemma \ref{lm:fgh}. 
Fix $\phi(x)$ obtained by applying this lemma to $\Sigma_2$-formula $\sigma(x)$.
We then have
\begin{align*}
\EA + \BSi_1 + \neg [1]_S\bot + \RFN_{[1]_S}(T) \vdash \Box_T \sigma(\num{x}) &\imp \Box_T (\sigma(\num{x}) \lor [1]_S \bot)\\
&\imp \Box_T [1]_S \phi(\num{x})\\
&\imp [1]_S \phi(\num{x})\\
&\imp \sigma(x) \lor [1]_S \bot\\
&\imp \sigma(x),
\end{align*}
i.e., $\EA + \BSi_1 + \neg [1]_S\bot \vdash \RFN_{[1]_S}(T) \imp \RFN_{\Sigma_2}(T)$, as required.
\end{proof}

The following theorem, which is a direct corollary of Lemma \ref{lm:rfnsig2} and Corollary \ref{col:cstrfn}, 
demonstrates that over a weaker theory, $\RFN_{\Sigma_2}(T)$ is equivalent
to the generalized Feferman theorem. 
\begin{theorem}\label{th:rfnsi2}
For each $\EA^+$-provably $\Sigma^0_2$-regular theory $T$ with $\Sigma^0_2$-ordinal $\alpha$
\begin{itemize}
\item[(i)] $\EA + \RFN_{\Sigma_2}(T) \vdash \forall \phi \left(\Box_{\Rfn(S)_{1 + \alpha}}\phi \imp [1]_S\phi\right)$,
\item[(ii)] $\EA +  \BSi_1 + \neg [1]_S\bot + \forall \phi \left(\Box_{\Rfn(S)_{1 + \alpha}}\phi \imp [1]_S\phi\right) \vdash \RFN_{\Sigma_2}(T)$.
\end{itemize}
\end{theorem}

In particular, taking $T$ and $S$ to be $\EA$ we obtain the following alternative axiomatization of the theory $\ISi_1$.
\begin{corollary}\label{cor:isi1rfn}
The following theories are deductively equivalent
\begin{itemize}
\item[(i)] $\ISi_1$,
\item[(ii)] $\EA + \RFN_{\Sigma_2}(\EA)$,
\item[(iii)] $\EA^+ + \BSi_1 + \forall \phi \left(\Box_{\EA + \Rfn(\EA)}\phi \imp [1]_\EA\phi\right)$.
\end{itemize}
\end{corollary} 
\begin{proof}
The equivalence $(i) \Leftrightarrow (ii)$ is well-known (cf. \cite[Theorem 7]{Bek05en}).
The equivalence $(ii) \Leftrightarrow (iii)$ follows from the previous corollary, since
$\ISi_1$ contains $\EA^+ + \BSi_1$ and $\EA + \neg [1]_\EA\bot \equiv \EA^+$.
\end{proof}

\subsection*{Acknowledgements}
The author would like to thank Lev D. Beklemishev for useful discussions and his comments on the draft of the paper.
The author would also like to thank Vladimir Yu. Shavrukov for informing us about his solution of the mentioned problems
and presenting his proof to us, which lead to the investigation of the formalizability of Feferman's theorem.

The reported study was funded by RFBR, project number 19-31-90050.

\end{document}